\documentclass[12pt,a4paper]{article}
\usepackage[utf8]{inputenc}
\usepackage{amsmath, amssymb, amsthm, verbatim, hyperref}
\usepackage{mathrsfs}
\usepackage[dvipsnames]{xcolor}
\usepackage{ragged2e}
\usepackage{marginnote}
\usepackage{enumerate}
\usepackage{makecell}
\usepackage{longtable}
\usepackage{epsfig}
\usepackage{tikz}
\usepackage{ytableau}
\usepackage{hyperref}
\usepackage{caption}
\usepackage{subcaption}
\hypersetup{hidelinks}
\usepackage{fancyvrb}
\usepackage[left=2.50cm, right=2.50cm, top=2.00cm, bottom=2.00cm]{geometry}
\usetikzlibrary {arrows.meta}

\theoremstyle{plain}

\newtheorem{innercustomthm}{Theorem}
\newenvironment{repthm}[1]
  {\renewcommand\theinnercustomthm{#1}\innercustomthm}
  {\endinnercustomthm}

\newtheorem{innercustomcor}{Corollary}
\newenvironment{repcor}[1]
  {\renewcommand\theinnercustomcor{#1}\innercustomcor}
  {\endinnercustomthm}

\newtheorem{thm}{Theorem}[section]
\newtheorem{prop}[thm]{Proposition}
\newtheorem{cor}[thm]{Corollary}
\newtheorem{lemma}[thm]{Lemma}

\theoremstyle{definition}
\newtheorem{definition}[thm]{Definition}
\newtheorem{example}[thm]{Example}

\theoremstyle{remark}

\newtheorem{remark}[thm]{Remark}

\newcommand{\RR}{\mathbb{R}}

\newcommand{\AAA}{\mathscr{A}}

\newcommand{\SSS}{\mathfrak{S}}
\newcommand{\des}{\operatorname{des}}
\newcommand{\Ass}{\textsf{Ass}}

\def\tand{\text{ and }}

\def\Tbad{T_{\text{bad}}}
\def\Tgood{T_{\text{good}}}

\def\TU{T_{U}}
\def\TL{T_{L}}

\def\rev{\text{rev}}
\def\DC{\text{DC}}
\def\DP{\text{DP}}
\def\DCw{\textsf{DC}}
\def\DPw{\textsf{DP}}
\def\DCl{\mathfrak{DC}}
\def\DPl{\mathfrak{DP}}
\def\append{\cdot}
\def\FT{\Tilde{F}}

\title{\textbf{An identity involving $h$-polynomials of poset associahedra and type B Narayana polynomials}}
\author{Son Nguyen\footnote{Department of Mathematics, Massachusetts Institute of Technology, Cambridge, MA 02139, USA. Email: \href{mailto:sonnvt@mit.edu}{sonnvt@mit.edu}}}
\date{\vspace{-2em}}

\begin{document}
\ytableausetup{centertableaux}

\maketitle

    \begin{abstract}
    For any finite connected poset $P$, Galashin introduced a simple convex $(|P|-2)$-dimensional polytope $\AAA(P)$ called the poset associahedron. Let $P$ be a poset with a proper autonomous subposet $S$ that is a chain of size $n$. For $1\leq i \leq n$, let $P_i$ be the poset obtained from $P$ by replacing $S$ by an antichain of size $i$. We show that the $h$-polynomial of $\AAA(P)$ can be written in terms of the $h$-polynomials of $\AAA(P_i)$ and type B Narayana polynomials. We then use the identity to deduce several identities involving Narayana polynomials, Eulerian polynomials, and stack-sorting preimages.
    \end{abstract}

\tableofcontents


\section{Introduction}\label{sec:intro}

    For a finite connected poset $P$, Galashin introduced the \textit{poset associahedron} $\AAA(P)$ (see \cite{galashin2021poset}). The faces of $\AAA(P)$ correspond to {\it tubings} of $P$, and the vertices of $\AAA(P)$ correspond to {\it maximal tubings} of $P$;  see Section~\ref{subsec:poset_ass} for the definitions. $\AAA(P)$ can also be described as a compactification of the configuration space of order-preserving maps $P \rightarrow \RR$. A realization of poset associahedra was given by Sack in \cite{sack2023realization}.
    
    Many polytopes can be described as poset associahedra, including permutohedra and associahedra. In particular, when $P$ is the claw poset, i.e. $P$ consists of a unique minimal element $0$ and $n$ pairwise-incomparable elements, then $\AAA(P)$ is the $n$-permutohedron. On the other hand, when $P$ is a chain of $n+1$ elements, i.e. $P = C_{n+1}$, then $\AAA(P)$ is the associahedron $K_{n+1}$.

    For a $d$-dimensional polytope $P$, the $f$-vector of $P$ is the sequence $(f_{0}(P),\ldots,f_{d}(P))$ where $f_i(P)$ is the number of $i$-dimensional faces of $P$. The $f$-polynomial of $P$ is
    \[ f_P(t) = \sum_{i=0}^{d} f_i(P) t^{i}. \]
    For simple polytopes such as poset associahedra, it is often better to consider the smaller and still nonnegative $h$-vector and $h$-polynomial defined by the relation
    \[ f_P(t) = h_P(t+1). \]

    We say $S$ is an \textit{autonomous subposet} of a poset $P$ if  $\text{for all } x, y \in S \tand z \in P-S, \text{ we have }$
    \[ (x \preceq z \Leftrightarrow y \preceq z) \tand (z \preceq x \Leftrightarrow z \preceq y). \]  
    In other words, every element in $P - S$ ``sees'' every element in $S$ the same. A subposet $S$ of $P$ is \textit{proper} if $S \neq P$. It was showed in \cite{nguyen2023poset} that the face numbers of poset associahedra is preserved under the flip operation of autonomous subposet. As a result, the face numbers of poset associahedra only depend on the comparability graph of the poset.

    In this paper, we pursue another question concerning autonomous subposets: what if we replace an autonomous subposet by another poset? It was conjectured in \cite[Conjecture 6.2]{nguyen2023stack} that the answer is particularly nice when we replace an autonomous subposet that is a chain $C_n$ by antichains $A_1,\ldots,A_n$. We will prove this conjecture in this paper.

    The type B Narayana polynomial is defined to be
    \[ B_n(x) = \sum_{k = 0}^{n-1}\binom{n-1}{k}^2x^k. \]
    For each permutation $w$, the cycle type of $w$ is a partition $\lambda(w) = (\lambda_1,\lambda_2,\ldots,\lambda_\ell)$. Then, we define $\ell_w = \ell$ to be the number of cycle in $w$, and
    \[ B_w(x) = B_{\lambda_1}(x)B_{\lambda_2}(x)\ldots B_{\lambda_3}(x). \]

    Our main theorem is the following.

    \begin{repthm}{\ref{thm:recurrence}}
        Let $P$ be a poset with a proper autonomous subposet $S$ that is a chain of size $n$. For $1\leq i \leq n$, let $P_i$ be the poset obtained from $P$ by replacing $S$ by an antichain of size $i$. Let $h_{P}(x)$, $h_{P_1}(x)$, $\ldots$, $h_{P_n}(x)$ be the $h$-polynomials of $\AAA(P)$, $\AAA(P_1)$, $\ldots$, $\AAA(P_n)$, respectively. Then,
        \begin{equation*}
            h_{P}(x) = \dfrac{1}{n!}\sum_{w\in \SSS_n} B_w(x)h_{P_{\ell_w}}(x).
        \end{equation*}
    \end{repthm}

    In particular, when $P$ is a chain $C_{n+1}$, $h_P(x)$ is the Narayana polynomial $N_n(x)$. Also, $P_i$ is the claw poset $A_{1} \oplus A_{i}$, where $\oplus$ denotes the ordinal sum, so $h_{P_i}(x)$ is the Eulerian polynomial $E_i(x)$. Thus, the following corollary is immediate from Theorem \ref{thm:recurrence}.

    \begin{repcor}{\ref{cor:identity1}}
        For all $n$,
        \[ N_n(x) = \dfrac{1}{n!}\sum_{w\in \SSS_n} B_w(x)E_{\ell_w}(x). \]
    \end{repcor}

    The outline of the paper is as follows. In Section \ref{sec:definition}, we will review relevant definitions of face numbers, poset associahedra, graph associahedra, and some families of polynomials. In Section \ref{sec:h-identity}, we will show that the main theorem follows from another identity that does not involve the $h$-vectors:
    \[ \sum_{w\in \SSS_n} t^{\ell_w}G_{w}(x) = \sum_{w\in\SSS_n}t(t+x)\ldots(t+(\ell_w -1)x)\FT_{w}(x). \]
    We refer the reader to Section \ref{subsec:key-polynomials} for the definitions of $G_{w}(x)$ and $\FT_{w}(x)$. Finally, we prove this identity in Section \ref{sec:proof-of-prop} and discuss some corollaries in Section \ref{sec:corollary}.

\vspace{-1em}

\section*{Acknowledgements}

I would like to thank my advisor Vic Reiner for introducing to me this topic and his amazing support. I would like to thank Andrew Sack for trying out many ideas with me. I would like to thank Colin Defant and Pavel Galashin for helpful conversations.

\section{Preliminaries}\label{sec:definition}

\subsection{Polytope and face numbers}\label{subsec:face_numbers}

    A \textit{convex polytope} $P$ is the convex hull of a finite collection of points in $\RR^n$. The \textit{dimension} of a polytope is the dimension of its affine span. A face $F$ of a convex polytope $P$ is the set of points in $P$ where some linear functional achieves its maximum on $P$. Faces that consist of a single point are called \textit{vertices} and $1$-dimensional faces are called \textit{edges} of $P$. A $d$-dimensional polytope $P$ is \textit{simple} if any vertex of $P$ is incident to exactly $d$ edges.

    For a $d$-dimensional polytope $P$, the \textit{face number} $f_i(P)$ is the number of $i$-dimensional faces of $P$. In particular, $f_{0}(P)$ counts the vertices and $f_1(P)$ counts the edges of $P$. The sequence $(f_{0}(P),f_1(P),\ldots,f_{d}(P))$ is called the \textit{$f$-vector} of $P$, and the polynomial
    \[ f_P(t) = \sum_{i=0}^{d} f_i(P) t^{i} \]
    is called the \textit{$f$-polynomial} of $P$. The \textit{$h$-vector} $(h_0(P),\ldots,h_d(P))$ and \textit{$h$-polynomial} $h_P(t) = \sum_{i=0}^dh_i(P)t^i$ are defined by the relation
    \[ f_P(t) = h_P(t+1). \]
    It is well-known that when $P$ is a simple polytope, its $h$-vector is nonnegative and satisfies the Dehn-Sommerville symmetry: $h_i(P) = h_{d-i}(P)$. When the $h$-polynomial is symmetric, recall that it has a unique expansion in terms of (centered) binomials $t^i(1+t)^{d-2i}$ for $0\leq i \leq d/2$. This unique expansion gives the \textit{$\gamma$-vector} $(\gamma_0(P),\ldots,\gamma_{\lfloor\frac{d}{2}\rfloor}(P))$ and \textit{$\gamma$-polynomial} $\gamma_P(t) = \sum_{i = 0}^{\lfloor\frac{d}{2}\rfloor}\gamma_i(P)t^i$ defined by 
    \[ h_P(t) = \sum_{i=0}^{\lfloor\frac{d}{2}\rfloor}\gamma_i(P)t^i(1+t)^{d-2i} = (1+t)^d\gamma_P\left( \dfrac{t}{(1+t)^2} \right).\]
    Note that the $\gamma$-vector may not be nonnegative.

\subsection{Poset associahedra}\label{subsec:poset_ass}

    We start with some poset terminologies.

    \begin{definition}\label{def:poset_def}
        Let $(P,\preceq)$ be a finite poset, and $\tau,\sigma \subseteq P$ be subposets.

        \begin{itemize}
            \item $\tau$ is \textit{connected} if it is connected as an induced subgraph of the Hasse diagram of $P$.

            \item $\tau$ is \textit{convex} if whenever $x,z\in \tau$ and $y \in P$ such that $x\preceq y \preceq z$, then $y\in \tau$.

            \item $\tau$ is a \textit{tube} of $P$ if it is connected and convex. $\tau$ is a \textit{proper tube} if $1 < |\tau| < |P|$.

            \item $\tau$ and $\sigma$ are \textit{nested} if $\tau \subseteq \sigma$ or $\sigma \subseteq \tau$. $\tau$ and $\sigma$ are \textit{disjoint} if $\tau \cap \sigma = \emptyset$.

            \item We say $\sigma \prec \tau$ if $\sigma \cap \tau = \emptyset$, and there exists $x\in \sigma$ and $y\in \tau$ such that $x\preceq y$.

            \item A \textit{tubing} $T$ of $P$ is a set of proper tubes such that any pair of tubes in $T$ is either nested or disjoint, and there is no subset $\{\tau_1,\tau_2,\ldots,\tau_k\}\subseteq T$ such that $\tau_1 \prec \tau_2 \prec \ldots \prec \tau_k \prec \tau_1$. We will refer to the latter condition as the \textit{acyclic condition}.

            \item A tubing $T$ is \textit{maximal} if it is maximal under inclusion, i.e. $T$ is not a proper subset of any other tubing.
        \end{itemize}
    \end{definition}

    \begin{example}
        Figure \ref{fig:tubingEx} shows examples and non-examples of tubings of posets. Note that the right-most example in Figure \ref{subfig:tubingNonEx} is a non-example since it violates the acyclic condition. In particular, if we label the tubes from right to left as $\tau_1,\tau_2,\tau_3$, then we have $\tau_1 \prec \tau_2 \prec \tau_3 \prec \tau_1$.

        \begin{figure}[h!]
         \centering
            \begin{subfigure}[b]{0.45\textwidth}
                \centering
                \includegraphics[scale = 0.5]{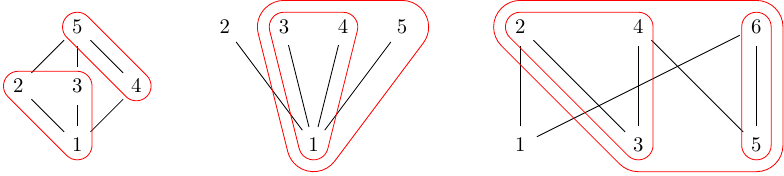}
                \caption{Examples}
                \label{subfig:tubingEx}
            \end{subfigure}
         \quad
            \begin{subfigure}[b]{0.45\textwidth}
                \centering
                \includegraphics[scale = 0.5]{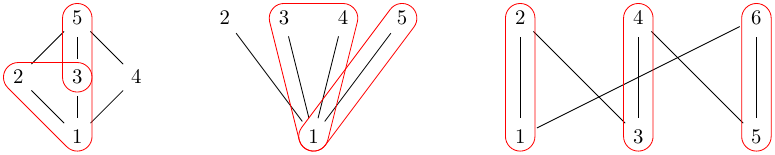}
                \caption{Non-examples}
                \label{subfig:tubingNonEx}
            \end{subfigure}
    
            \caption{Examples and non-examples of tubings of posets}
            \label{fig:tubingEx}
        \end{figure}

    \vspace{-1em}
    \end{example}

    \begin{definition}[{\cite[Theorem 1.2]{galashin2021poset}}]
        For a finite connected poset $P$, there exists a simple, convex polytope $\AAA(P)$ of dimension $|P|-2$ whose face lattice is isomorphic to the set of tubings ordered by reverse inclusion. The faces of $\AAA(P)$ correspond to tubings of $P$, and the vertices of $\AAA(P)$ correspond to maximal tubings of $P$. This polytope is called the \textbf{poset associahedron} of $P$.
    \end{definition}

    \begin{example}
        Examples of poset associahedra can be seen in Figure \ref{fig:posetAssEx}. In particular, if $P$ is a claw, i.e. $P$ consists of a unique minimal element $0$ and $n$ pairwise-incomparable elements as shown in Figure \ref{subfig:posetExPerm}, $\AAA(P)$ is a permutohedron. If $P$ is a chain, $\AAA(P)$ is an associahedron.

        \begin{figure}[h!]
         \centering
            \begin{subfigure}[c]{0.45\textwidth}
                \centering
                \includegraphics[scale = 0.3]{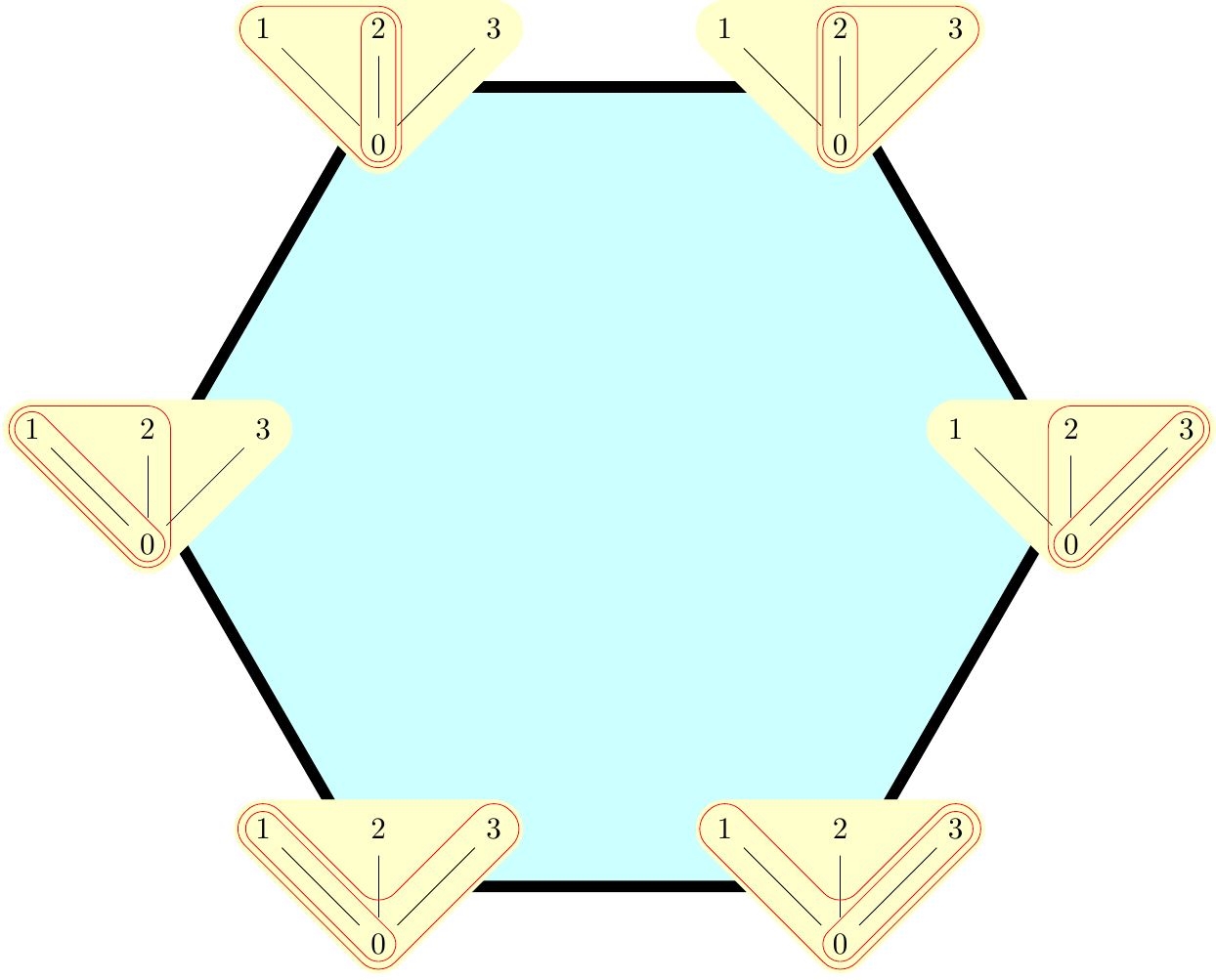}
                \caption{Permutohedron}
                \label{subfig:posetExPerm}
            \end{subfigure}
         \quad
            \begin{subfigure}[c]{0.45\textwidth}
                \centering
                \includegraphics[scale = 0.3]{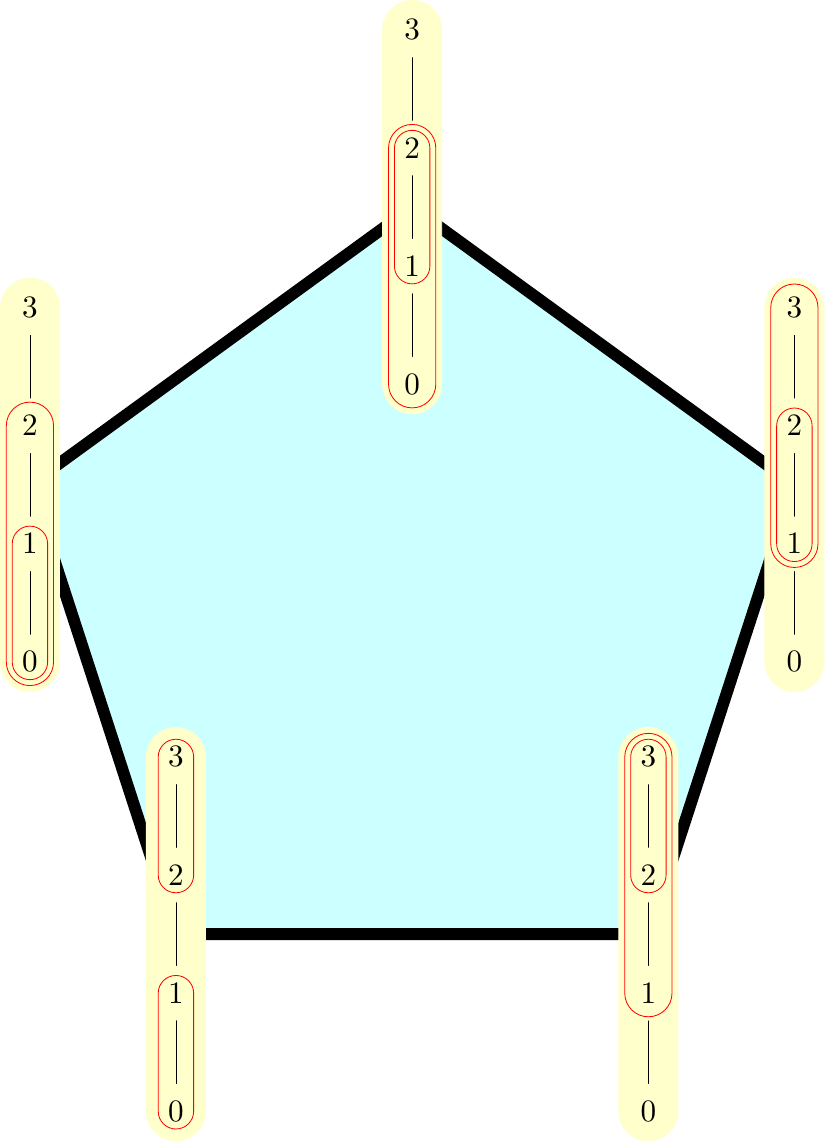}
                \caption{Associahedron}
                \label{subfig:posetExAss}
            \end{subfigure}
    
            \caption{Permutohedron and associahedron as poset associahedra}
            \label{fig:posetAssEx}
        \end{figure}
    \end{example}

\subsection{Graph associahedra}\label{subsec:graph_ass}

    Graph associahedra are generalized permutohedra arising as special cases of nestohedra. We refer the readers to \cite{postnikov2006faces} for a comprehensive study of face numbers of generalized permutohedra and nestohedra.

    \begin{definition}\label{def:graph_def}
        Let $G = (V,E)$ be a graph, and $\tau,\sigma\subseteq V$ be subsets of vertices.

        \begin{itemize}
            \item $\tau$ is a \textit{tube} of $G$ if $\tau \neq V$ and it induces a connected subgraph of $G$.
            
            \item $\tau$ and $\sigma$ are \textit{nested} if $\tau \subseteq \sigma$ or $\sigma \subseteq \tau$. $\tau$ and $\sigma$ are \textit{disjoint} if $\tau \cap \sigma = \emptyset$.

            \item $\tau$ and $\sigma$ are \textit{compatible} if they are nested or they are disjoint and $\tau \cup \sigma$ is not a tube.

            \item A \textit{tubing} $T$ of $G$ is a set of pairwise compatible tubes.

            \item A tubing $T$ is \textit{maximal} if it is maximal by inclusion, i.e. $T$ is not a proper subset of any other tubing.
        \end{itemize}
    \end{definition}

    Figure \ref{fig:graphTubingEx} shows examples and non-examples of tubings of graphs. Note that the left-most example in Figure \ref{subfig:graphTubingNonEx} is a non-example since the tubes $\{1\}$ and $\{4\}$ are disjoint yet their union $\{1,4\}$ is still a tube. The same reason applies for the right-most example.

    \begin{figure}[h!]
     \centering
        \begin{subfigure}[b]{0.45\textwidth}
            \centering
            \includegraphics[scale = 0.5]{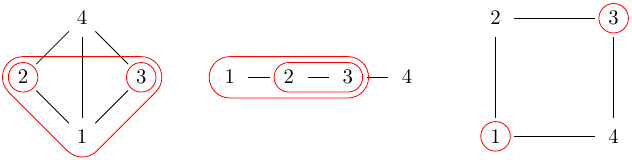}
            \caption{Examples}
            \label{subfig:graphTubingEx}
        \end{subfigure}
     \quad
        \begin{subfigure}[b]{0.45\textwidth}
            \centering
            \includegraphics[scale = 0.5]{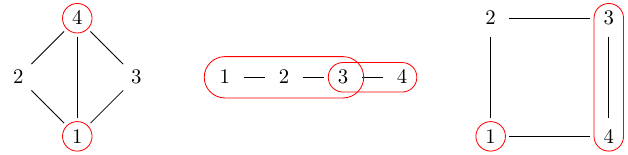}
            \caption{Non-examples}
            \label{subfig:graphTubingNonEx}
        \end{subfigure}

        \caption{Examples and non-examples of tubings of graphs}
        \label{fig:graphTubingEx}
    \end{figure}

    \begin{definition}
        For a connected graph $G = (V,E)$, the \textbf{graph associahedron} of $G$ is a simple, convex polytope $\Ass(G)$ of dimension $|V|-1$ whose face lattice is isomorphic to the set of tubings ordered by reverse inclusion. The faces of $\Ass(G)$ correspond to tubings of $G$, and the vertices of $\Ass(G)$ correspond to maximal tubings of $G$.
    \end{definition}

    Examples of graph associahedra can be seen in Figure \ref{fig:graphAssEx}. In particular, if $G$ is a complete graph, $\Ass(G)$ is a permutohedron. If $G$ is a path graph, $\Ass(G)$ is an associahedron.

    \begin{figure}[h!]
     \centering
        \begin{subfigure}[c]{0.45\textwidth}
            \centering
            \includegraphics[scale = 0.3]{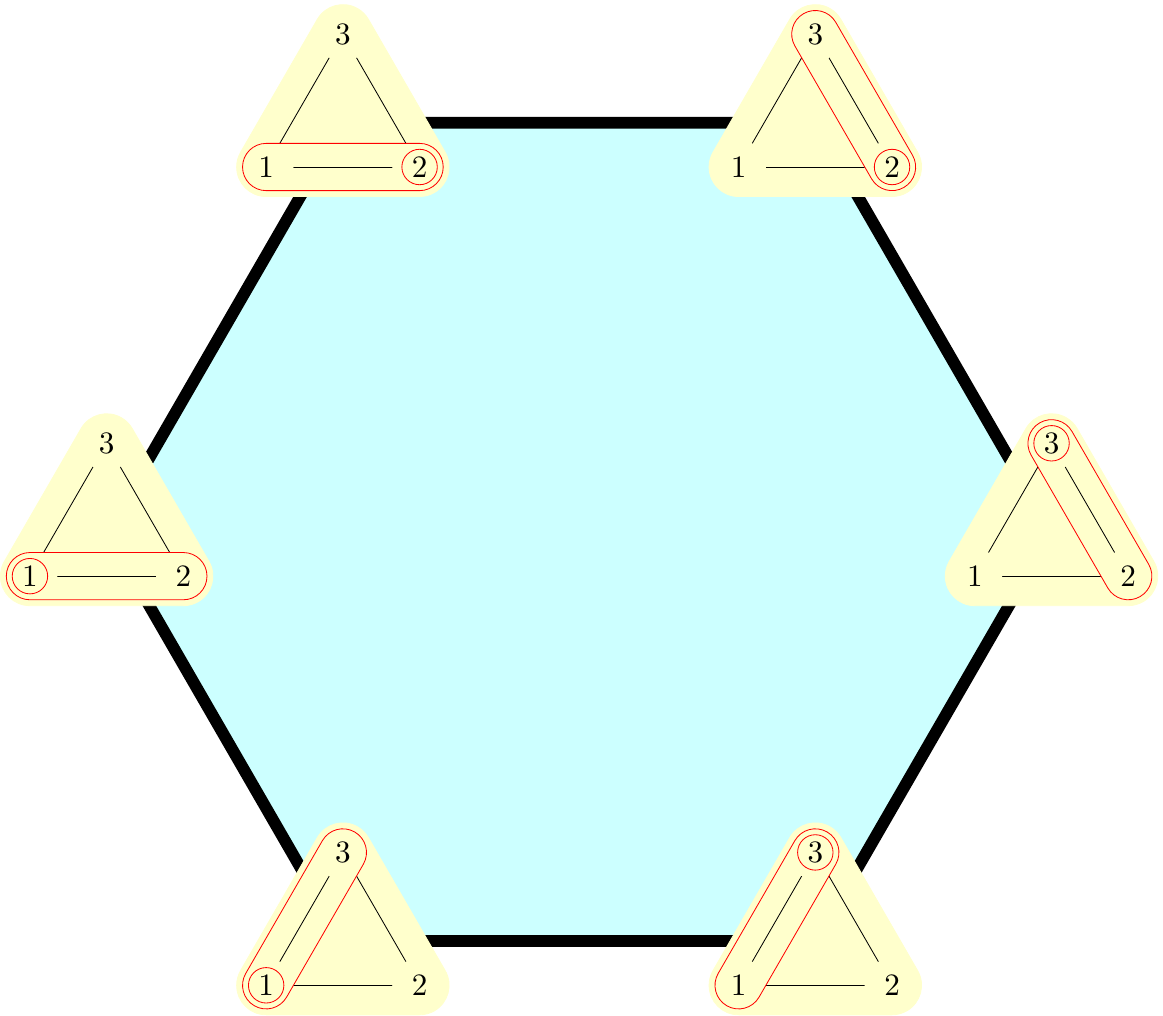}
            \caption{Permutohedron}
            \label{subfig:graphExPerm}
        \end{subfigure}
     \quad
        \begin{subfigure}[c]{0.45\textwidth}
            \centering
            \includegraphics[scale = 0.3]{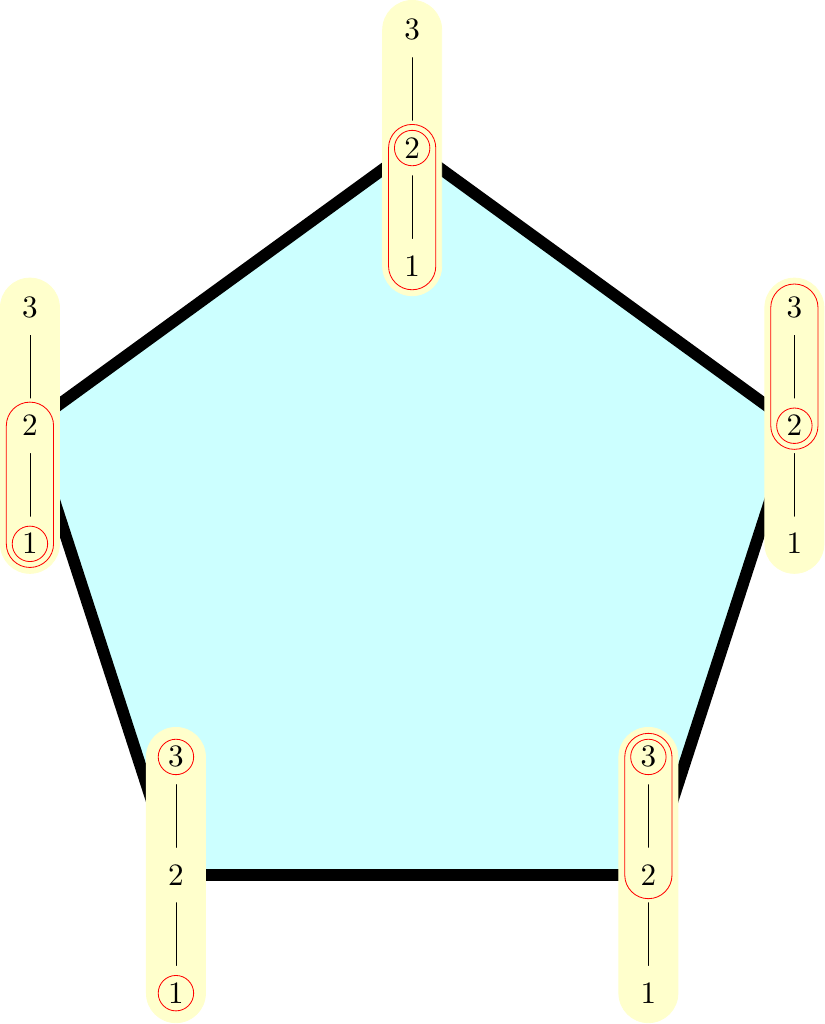}
            \caption{Associahedron}
            \label{subfig:graphExAss}
        \end{subfigure}

        \caption{Permutohedron and associahedron as graph associahedra}
        \label{fig:graphAssEx}
    \end{figure}

    \begin{remark}
        In Section \ref{sec:proof-of-prop}, we will work with tubings of directed graphs. When constructing tubings for directed graphs, we will ignore the directions of the edges and treat the graphs as undirected.
    \end{remark}

\subsection{Polynomials}\label{subsec:key-polynomials}

    Let us now introduce some relevant polynomials. The (type A) Narayana polynomial is defined to be
    \[ N_n(x) = \sum_{k = 0}^{n-1}\dfrac{1}{n}\binom{n}{k}\binom{n}{k+1}x^k. \]
    For example, we have
    \begin{align*}
        N_1(x) &= 1, \\
        N_2(x) &= 1 + x, \\
        N_3(x) &= 1 + 3x + x^2, \\
        N_4(x) &= 1 + 6x + 6x^2 + x^3.
    \end{align*}
    It is well-known that Narayana polynomials give the $h$-vectors of the classical associahedra. Recall that the classical associahedra is also the graph associahedra of path graphs. The corresponding $f$-vectors are
    \[ F_n(x) = N_n(x+1). \]
    For example, we have
    \begin{align*}
        F_1(x) &= 1, \\
        F_2(x) &= 2 + x, \\
        F_3(x) &= 5 + 5x + x^2, \\
        F_4(x) &= 14 + 21x + 9x^2 + x^3.
    \end{align*}
    We also define
    \[ \FT_n(x) = nF_{n-1}(x), \]
    with the convention that $F_0(x) = 1$. For example, we have
    \begin{align*}
        \FT_1(x) &= 1, \\
        \FT_2(x) &= 2, \\
        \FT_3(x) &= 6 + 3x, \\
        \FT_4(x) &= 20 + 20x + 4x^2.
    \end{align*}
    Similarly, the type B Narayana polynomial is defined to be
    \[ B_n(x) = \sum_{k = 0}^{n-1}\binom{n-1}{k}^2x^k. \]
    For example, we have
    \begin{align*}
        B_1(x) &= 1, \\
        B_2(x) &= 1 + x, \\
        B_3(x) &= 1 + 4x + x^2, \\
        B_4(x) &= 1 + 9x + 9x^2 + x^3.
    \end{align*}
    The type B Narayana polynomials show up as the rank-generating function of the type B analogue $\text{NC}^B_n$ of the lattice of non-crossing partitions (see \cite{reiner1997non}) and the $h$-polynomials of type B associahedra (see \cite{simion2003type}). Notably, type B associahedra are also graph associahedra of cycle graphs (see \cite{postnikov2006faces}). The sum of the coefficients in $B_{n+1}(x)$ is $\binom{2n}{n}$, which is called type B Catalan number. The corresponding $f$-vectors of type B associahedra are
    \[ G_n(x) = B_n(x+1). \]
    For example, we have
    \begin{align*}
        G_1(x) &= 1, \\
        G_2(x) &= 2 + x, \\
        G_3(x) &= 6 + 6x + x^2, \\
        G_4(x) &= 20 + 30x + 12x^2 + x^3.
    \end{align*}
    For each family of polynomials $\{P_n(x)\}$, we define $P^\rev_n(x)$ to be the palindromic polynomial
    \[ P^\rev_n(x) = x^{n-1}P_n\left(\dfrac{1}{x}\right). \]
    We will make use of these palindromic polynomials because in the case of $\{F_n(x)\}$ and $\{G_n(x)\}$, $\{F^\rev_n(x)\}$ and $\{G^\rev_n(x)\}$ count tubings by the number of tubes.
    
    In addition, for each family of polynomials $\{P_n(x)\}$, and each partition $\lambda= (\lambda_1, \lambda_2,\ldots,\lambda_\ell)$, we define
    \[ P_{\lambda}(x) = P_{\lambda_1}(x)P_{\lambda_2}(x)\ldots P_{\lambda_\ell}(x). \]
    For example, we have
    \begin{align*}
        N_{(4,2,1)}(x) &= (1 + 6x + 6x^2 + x^3)(1 + x)(1), \\
        F_{(4,2,1)}(x) &= (14 + 21x + 9x^2 + x^3)(2 + x)(1), \\
        \FT_{(4,2,1)}(x) &= (20 + 20x + 4x^2)(2)(1), \\
        B_{(4,2,1)}(x) &= (1 + 9x + 9x^2 + x^3)(1 + x)(1), \\
        G_{(4,2,1)}(x) &= (20 + 30x + 12x^2 + x^3)(2 + x)(1).
    \end{align*}
    For each permutation $w$, the cycle type of $w$ is a partition $\lambda(w)$, and the number of cycles in $w$ is $\ell_w = \ell(\lambda(w))$. We abuse notation and define
    \[ P_w(x) = P_{\lambda(w)}(x) \]
    for each family of polynomials $\{P_n(x)\}$. Note that this means $P_{w_1}$ and $P_{w_2}$ are the same if $w_1$ and $w_2$ are in the same conjugacy class.

    Finally, we denote by $s_{n,k}$ the \textit{unsigned Stirling number of the first kind}, which counts the number of permutations of $\SSS_n$ with $k$ cycles. Note that $s_{n,k}$ is the coefficient of $x^k$ in $x(x+1)\ldots(x+n-1)$, or equivalently the coefficient of $x^{n-k}$ in $1(1+x)\ldots(1 + (n-1)x)$.

\section{The main theorem}\label{sec:h-identity}
    
    Recall that for a poset $P$, a subposet $S$ of $P$ is called \emph{autonomous} if $\text{for all } x, y \in S \tand z \in P-S, \text{ we have }$
    \[ (x \preceq z \Leftrightarrow y \preceq z) \tand (z \preceq x \Leftrightarrow z \preceq y). \]
    A subposet $S$ of $P$ is \textit{proper} if $S \neq P$. Our main theorem is the following.
    \begin{thm}\label{thm:recurrence}
        Let $P$ be a poset with a proper autonomous subposet $S$ that is a chain of size $n$. For $1\leq i \leq n$, let $P_i$ be the poset obtained from $P$ by replacing $S$ by an antichain of size $i$. Let $h_{P}(x)$, $h_{P_1}(x)$, $\ldots$, $h_{P_n}(x)$ be the $h$-polynomials of $\AAA(P)$, $\AAA(P_1)$, $\ldots$, $\AAA(P_n)$, respectively. Then,
        \begin{equation}\label{eqn:recurrence}
            h_{P}(x) = \dfrac{1}{n!}\sum_{w\in \SSS_n} B_w(x)h_{P_{\ell_w}}(x).
        \end{equation}
    \end{thm}
    For example, when $n = 2$, we have
    \[ h_P(x) = \dfrac{1}{2}\left( h_{P_2}(x) + (1+x)h_{P_1}(x) \right). \]
    We will show that Theorem \ref{thm:recurrence} follows from the following proposition.
    \begin{prop}\label{prop:better_identity}
        For all $n$,
        \begin{equation}\label{eqn:better_identity}
            \sum_{w\in \SSS_n} t^{\ell_w}G_{w}(x) = \sum_{w\in\SSS_n}t(t+x)\ldots(t+(\ell_w -1)x)\FT_{w}(x).
        \end{equation}
    \end{prop}
    \begin{example}
        For $n = 3$, the LHS of (\ref{eqn:better_identity}) is
        \[ t^3 + 3t^2(x+2) + 2t(x^2 + 6x + 6), \]
        and the RHS is
        \[ t(t+x)(t+2x) + 3t(t+x)(2) + 2t(3x+6). \]
        One can check that they are equal.
    \end{example}
    \begin{prop}\label{prop:reduce}
        Theorem \ref{thm:recurrence} follows from Proposition \ref{prop:better_identity}.
    \end{prop}
    We will need a few lemmas to prove Proposition \ref{prop:reduce}. Let $P$ be a poset with a proper autonomous subposet $S = C_n$. We say a tubing $T$ of $P$ is \textit{degradable} if there is a tube $\tau\in T$ such that $\tau\subseteq S$. We say that $T$ is \textit{non-degradable} otherwise. Our main lemma is the following, which will be proved in Section \ref{subsec:proof-of-main-lem}.

    \begin{lemma}\label{lem:non-degradable}
        Let $P, P_1,\ldots,P_n$ be defined as in Theorem \ref{thm:recurrence}. Let $t_{k}$ be the number of non-degradable tubings of $P$ with $k$ tubes, and $t_{i,k}$, for $1\leq i\leq n$, be the number of tubings of $P_i$ with $k$ tubes. Then
        \[ n!t_{k} = \sum_{i = 1}^n s_{n,i}t_{i,k}. \]
    \end{lemma}

    On the other hand, if $T$ is degradable, we say a tube $\tau$ of $T$ is \textit{degrading} if $\tau\subseteq S$. Clearly, the degrading tubes of $T$ gives a tubing of $S$. Here we modify the rule slightly and allow $S$ to be a tube of $S$.

    Given a tubing of $S = C_n$, we say a tube is \textit{maximal} if it is not contained in another tube. We say an element $s\in S$ is \textit{lonely} if it is not contained in any tube. Then, the lonely elements and maximal tubes of each tubing gives a composition of $n$.
    
    \begin{example}
        Figure \ref{fig:degrading} shows a tubing of $S = C_{10}$. The lonely elements and maximal tubes are colored red. The composition is $(2,1,1,3,3)$.

        \begin{figure}[h!]
            \centering
            \includegraphics[scale = 0.7]{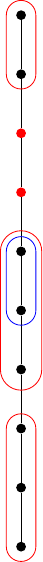}
            \caption{A tubing of $S = C_{10}$}
            \label{fig:degrading}
        \end{figure}
    \end{example}

    The following lemma is immediate.

    \begin{lemma}\label{lem:degradable}
        Let $P$ be defined as in Theorem \ref{thm:recurrence}. Let $T'$ be a tubing of $S = C_n$, let $(\alpha_1,\ldots,\alpha_\ell)$ be the composition corresponding to $T'$. Then the number of tubings with $k$ tubes of $P$ that contain $T'$ is the same as the number of non-degradable tubings of $P'$ with $k-|T'|$ tubes, where $P'$ is obtained from $P$ by replacing $S$ by $C_\ell$.
    \end{lemma}

    To see Lemma \ref{lem:degradable}, from a tubing with $k$ tubes of $P$ that contain $T'$, one can contract every maximal tube of $T'$ into a single element and obtain a non-degradable tubing of $P'$ with $k-|T'|$ tubes. Figure \ref{fig:degrade-to-nondegrade} gives an example of this contraction.

    \begin{figure}[h!]
        \centering
        \includegraphics[scale = 0.8]{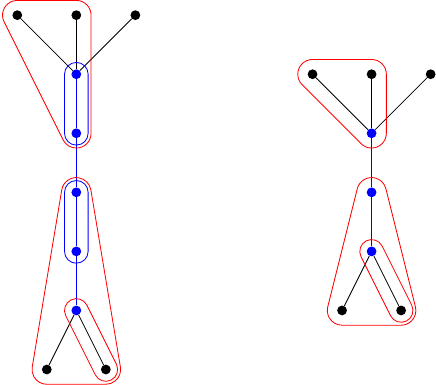}
        \caption{A degradable tubing of $P$ with $S = C_5$ (left) and a non-degradable tubing of $P'$ with $S$ replaced by $C_3$ (right)}
        \label{fig:degrade-to-nondegrade}
    \end{figure}

    Combining Lemma \ref{lem:non-degradable} and \ref{lem:degradable}, we have the following lemma.

    \begin{lemma}\label{lem:f-sum}
        With the same notations as in Theorem \ref{thm:recurrence}, let $f_{P}(x)$, $f_{P_1}(x)$, $\ldots$, $f_{P_n}(x)$ be the $h$-polynomials of $\AAA(P)$, $\AAA(P_1)$, $\ldots$, $\AAA(P_n)$, respectively. Then,
        \begin{equation}\label{eqn:f-sum}
            n!f_P(x) = \sum_{\lambda\vdash n}\dfrac{n!}{\ell(\lambda)!}R(\lambda)F_{(\lambda_1 - 1,\lambda_2 - 1,\ldots)} (x) \left(\sum_{k = 1}^{\ell(\lambda)} s_{\ell(\lambda), k}x^{\ell(\lambda) - k} f_{P_k}(x) \right), 
        \end{equation}
        where $R(\lambda)$ is the number of rearrangements of $\lambda$.
    \end{lemma}

    \begin{proof}
        For each composition $\alpha = (\alpha_1,\ldots,\alpha_\ell)$ that is a rearrangement of a partition $\lambda$, the generating function for the degrading tubings of $S$ whose composition is $\alpha$ is $F_{(\lambda_1 - 1,\lambda_2 - 1,\ldots)}(x)$. This is because for each maximal tube $\tau$ of $S$, the tubes contained in $\tau$ form a tubing of $C_{|\tau|}$, and the generating function for such tubings is $F_{|\tau| -1}(x)$. By Lemma \ref{lem:degradable}, degradable tubings of $P$ in which the composition of the degrading tubings is $\alpha$ can be viewed as non-degradable tubings of $P'$ obtained from $P$ by replacing $S$ by $C_\ell$. Then by Lemma \ref{lem:non-degradable}, non-degradable tubings of $P'$ can be written as a sum of tubings of $P_1,\ldots,P_\ell$ with coefficients $s_{\ell,k}$. This gives the desired formula.
    \end{proof}

    Now we are ready to prove Proposition \ref{prop:reduce}.

    \begin{proof}[Proof of Proposition \ref{prop:reduce}]
        For each partition $\lambda$, one can view $\lambda$ as a tuple $(c_1,\ldots,c_n)$ such that $\sum_i ic_i = n$. Then, in the RHS of (\ref{eqn:f-sum}),
        \[ R(\lambda) = \dfrac{\ell(\lambda)!}{c_1!\ldots c_n!}. \]
        Thus,
        \[ \dfrac{n!}{\ell(\lambda)!}R(\lambda)F_{(\lambda_1 - 1,\lambda_2 - 1,\ldots)} = \dfrac{n!}{c_1!\ldots c_n!}F_{(\lambda_1 - 1,\lambda_2 - 1,\ldots)}(x) \]
        \[ = \dfrac{n!}{\lambda_1\ldots\lambda_\ell\cdot c_1!\ldots c_n!}\lambda_1\ldots\lambda_\ell\cdot F_{(\lambda_1 - 1,\lambda_2 - 1,\ldots)}(x) = \dfrac{n!}{\lambda_1\ldots\lambda_\ell\cdot c_1!\ldots c_n!}\FT_{\lambda}(x). \]
        Notice that $\frac{n!}{\lambda_1\ldots\lambda_\ell\cdot c_1!\ldots c_n!}$ is the number of permutations in $\SSS_n$ with cycle type $\lambda$, so the RHS of (\ref{eqn:f-sum}) becomes
        \[ \sum_{w\in\SSS_n}\FT_{w}(x)\left(\sum_{k = 1}^{\ell_w} s_{\ell_w, k}x^{\ell_w - k} f_{P_k}(x) \right). \]
        Recall that $s_{n,k}$ is the coefficient of $x^{n-k}$ in $1(1+x)\ldots(1 + (n-1)x)$. Hence, the coefficient of $f_{P_k}(x)$ in the above sum is the coefficient of $t^k$ in
        \[ \sum_{w\in\SSS_n}\FT_{w}(x)t(t+x)\ldots(t+(\ell_w -1)x), \]
        which is the RHS of (\ref{eqn:better_identity}).

        Finally, by the $h$-to-$f$-vector conversion, one can check that the coefficient of $f_{P_k}(x)$ in the RHS of (\ref{eqn:recurrence}) is the coefficient of $t^k$ in the LHS of (\ref{eqn:better_identity}). Hence, Propostion \ref{prop:better_identity} implies Theorem \ref{thm:recurrence}.
    \end{proof}

\subsection{Proof of Lemma \ref{lem:non-degradable}}\label{subsec:proof-of-main-lem}

\subsubsection{Decomposition}\label{subsubsec:decom}

    One main part of the proof involves constructing tubings of $P_1,\ldots,P_n$ from non-degradable tubings of $P$. To do this, we will use the decomposition in \cite{nguyen2023poset}.

    \begin{definition}\label{def:good-bad-tubes}
        A tube $\tau \in T$ is \emph{good} if $\tau \subseteq P - S$, $\tau \subseteq S$, or $S \subseteq \tau$ and is \emph{bad} otherwise.  We denote the set of good tubes by $\Tgood$ and the set of bad tubes by $\Tbad$. A tube $\tau \in \Tbad$ is called \emph{lower} (resp. \emph{upper}) if there exist $x \in \tau-S$ and $y \in \tau \cap S \text{ such that } x \preceq y$ (resp. $y \preceq x$). We denote the set of lower tubes by $\TL$ and the set of upper tubes by $\TU$.
    \end{definition}

    \begin{lemma}[{\cite[Lemma 3.4]{nguyen2023poset}}]
        $\Tbad$ is the disjoint union of $\TL$ and $\TU$. Moreover, $\TL$ and $\TU$ each form a nested sequence.
    \end{lemma}

    \begin{definition}[Tubing decomposition]  
    \label{def:tubing_decomp}
    
        Since $\TL$ forms a nested sequence, we can write $\TL = \{\tau_1, \tau_2, \ldots \}$ where $\tau_i \subset \tau_{i+1}$ for all $i$.  For convenience, we define $\tau_0 = \emptyset$. We define a nested sequence $\mathcal L = (L_1, L_2, \ldots)$ and a sequence of disjoint sets $\mathcal M_L = (M_{L}^1, M_{L}^2, \ldots)$ as follows. 
        \begin{itemize}

        \item For each $i \ge 1$, let $L_i = \tau_i - S$, and mark $L_i$ with a star if $(\tau_i - \tau_{i-1}) \cap S \neq \emptyset$.  
    
    \item If $L_{i}$ is the $j$-th starred set, let $M_{L}^j = (\tau_{i}-\tau_{i-1}) \cap S$.
    
    \end{itemize}
    We define the sequences $\mathcal U$ and $\mathcal M_U$ analogously.  
    We make the following definitions.
    \begin{itemize}
        
        \item Let $\hat M := S - \bigcup\limits_{\tau \in \Tbad} \tau.$
    
        \item For sequences $\mathbf{a}=(a_1,\ldots,a_p) \text{ and } \mathbf{b}=(b_1,\ldots,b_q)$, let the sequence $\mathbf{a} \append \mathbf{b}:=(a_1,\ldots,a_p,b_1,\ldots,b_q)$ be their concatenation, and let $\overline{\mathbf a}:=(a_p,\ldots,a_1)$ be the reverse of $\mathbf{a}$.
        \item We define $$\mathcal M :=  \begin{cases}
            \mathcal M_L \append \overline{\mathcal M}_U & \text{if } \hat M = \emptyset\\
    
            \mathcal M_L \append (\hat M) \append \overline{\mathcal M}_U & \text{if } \hat M \neq \emptyset 
            
        \end{cases}$$
        where $(\hat M)$ is the sequence containing exactly one set: $\hat M$.
    
    \item The \emph{decomposition} of $\Tbad$ is the triple $(\mathcal L, \mathcal M, \mathcal U)$.
    \end{itemize}
    
    \end{definition}

    \begin{example}
        Figure \ref{fig:decomposition} gives an example of a decomposition.

        \begin{figure}[h!]
            \centering
            \begin{subfigure}[t]{0.33\textwidth}
                \centering
                \includegraphics[scale = 1]{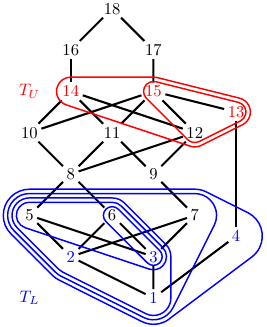}
                \caption{$T_L$ is blue and $T_U$ is red.}
                \label{subfig:TLandTU}
            \end{subfigure}~
            \begin{subfigure}[t]{0.66\textwidth}
                \centering
                \includegraphics[scale = 1]{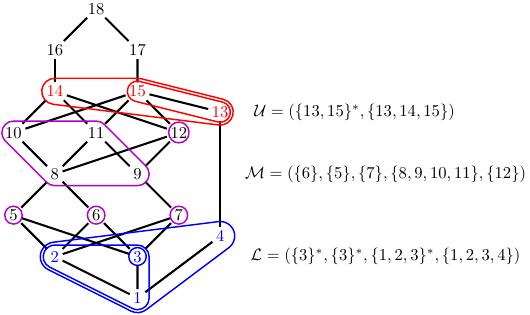}
                \caption{$\mathcal L$ is blue, $\mathcal M$ is purple, and $\mathcal U$ is red.}
                \label{subfig:Decomposition}
            \end{subfigure}
            \caption{The decomposition of $\Tbad$.}
            \label{fig:decomposition}
        \end{figure}
    \end{example}

    \begin{lemma}[{\cite[Lemma 3.6]{nguyen2023poset}}]\label{lem:reconstruct}
        $\Tbad$ can be reconstructed from its decomposition.
    \end{lemma}

\subsubsection{Proof}\label{subsubsec:proof-main-lem}

    In order to prove Lemma \ref{lem:non-degradable}, we will need a small bijection between
    \begin{itemize}
        \item pairs $(w,\alpha)$ where $w\in\SSS_n$ and $\alpha$ is a composition of $n$ into $k$ parts, and
        \item pairs $(\omega,U)$ where $\omega$ is a permutation in $\SSS_n$ with $\ell$ cycles and $U$ is an ordered set partition of $\{1,\ldots,\ell\}$ into $k$ sets.
    \end{itemize}
    Our bijection is constructed as follows. Given a pair $(w,\alpha)$ where $w\in\SSS_n$ and $\alpha = (\alpha_1,\ldots,\alpha_k)$ is a composition of $n$ into $k$ parts:
    \begin{enumerate}
        \item Let $\mu_i = w_{\alpha_1+\ldots+\alpha_{i-1}+1}\ldots w_{\alpha_1+\ldots+\alpha_i}$.
        \item Let $V_i = \{v_1<\ldots<v_{\alpha_i}\}$ be the set of elements in $\mu_i$, then we can consider $\mu_i$ as a permutation of the elements $v_1,\ldots,v_{\alpha_i}$. Let $\sigma_i$ be the cycle decomposition of this permutation.
        \item Let $\omega = \sigma_1\ldots\sigma_i$, this is the desired permutation.
        \item Order the cycles in $\omega$ as $\nu_1, \ldots, \nu_\ell$ in the order of their smallest element, then let $U_i = \{j~|~\text{$\sigma_i$ contains $\nu_j$}\}$. $(U_1,\ldots,U_k)$ is the desired ordered set partition.
    \end{enumerate}
    \begin{example}
        Let $w = 965347128$ and $\alpha = (3,4,2)$.
        \begin{enumerate}
            \item $\mu_1 = 965, \mu_2 = 3471, \mu_{3} = 28$.
            \item $\sigma_1 = (59)(6)$ when we consider $965$ as a permutation of $569$; similarly, $\sigma_2 = (1347), \sigma_3 = (2)(8)$.
            \item $\omega = (59)(6)(1347)(2)(8) = 324796185$.
            \item The cycles are ordered as $\nu_1 = (1347), \nu_2 = (2), \nu_3 = (59), \nu_4 = (6), \nu_5 = (8)$, then $U_1 = \{3,4\}$ since $\sigma_1$ contains $\nu_3$ and $\nu_4$; similarly, $U_2 = \{1\}, U_3 = \{2,5\}$.
        \end{enumerate}
    \end{example}
    \begin{proof}[Proof of Lemma \ref{lem:non-degradable}]
        We will construct a bijection between 
        \begin{itemize}
            \item pairs $(w,T)$ where $w\in\SSS_n$ and $T$ is a non-degradable tubing of $P$ with $k$ tubes, and
            \item pairs $(\omega,T')$ where $\omega$ is a permutation in $\SSS_n$ with $\ell$ cycles and $T'$ is a tubing of $P_\ell$ with $k$ tubes.
        \end{itemize}
        Our construction of $T'$ from $T$ uses the decomposition in Section \ref{subsubsec:decom}. We denote the set of good tubes of $T$ by $\Tgood$ and the set of bad tubes of $T$ by $\Tbad$. In this case, we do not have tubes $\tau \subseteq S$ because $T$ is non-degradable. Hence, we can keep $\Tgood$ for $T'$. Then, we decompose $\Tbad$ into a triple $(\mathcal L, \mathcal M, \mathcal U)$ where $\mathcal L$ and $\mathcal U$ are nested sequences of sets, some of which may be marked, contained in $P-S$ and $\mathcal M$ is an ordered set partition of $S = C_n$. Finally, we construct $\Tbad'$ from a triple $(\mathcal L, \mathcal M', \mathcal U)$, where $\mathcal M'$ is an ordered set partition of some $A_\ell$ and $|\mathcal{M}'| = \mathcal{M}$, and have $T' = \Tgood \sqcup \Tbad'$. Hence, our bijection between $(w,T)$ and $(\omega,T')$ comes down to a bijection between $(w,\mathcal{M})$ and $(\omega, \mathcal{M}')$, where $|\mathcal{M}'| = \mathcal{M}$.

        Since $S = C_n$, there is an easy one-to-one correspondence between sequences $\mathcal{M}$ of $S$ and compositions $\alpha$ of $n$. On the other hand, any ordered set partition of $A_\ell$ is an ordered set partition $U$ of $\{1,\ldots,\ell\}$. Therefore, a bijection between $(w,\mathcal{M})$ and $(\omega,\mathcal{M}')$, where $|\mathcal{M}'| = \mathcal{M}$, is essentially a bijection between $(w,\alpha)$ and $(\omega, U)$, which is the bijection discussed at the beginning of the section.
    \end{proof}

\section{Proof of Proposition \ref{prop:better_identity}}\label{sec:proof-of-prop}

    Recall that Proposition \ref{prop:better_identity} states that for all $n$,
    \begin{equation*}
        \sum_{w\in \SSS_n} t^{\ell_w}G_{w}(x) = \sum_{w\in\SSS_n}t(t+x)\ldots(t+(\ell_w -1)x)\FT_{w}(x).
    \end{equation*}
    To begin our proof, let us rewrite equation (\ref{eqn:better_identity}) slightly. Let $(\lambda_1,\ldots,\lambda_\ell)$ be the cycle type of a permutation $w\in\SSS_n$. Then
    \[ \dfrac{t^{\ell_w}G_{w}(x)}{x^n} = \dfrac{t^{\ell_w}}{x^{\ell_w}}\cdot\dfrac{G_{\lambda_1}(x)}{x^{\lambda_1 - 1}}\cdots\dfrac{G_{\lambda_\ell}(x)}{x^{\lambda_\ell - 1}}. \]
    Replacing $\dfrac{t}{x}$ by $t$ and $x$ by $\dfrac{1}{x}$, then this becomes
    \[ t^{\ell_w}G^\rev_w(x). \]
    Similarly, after dividing by $x^n$ then replacing $\dfrac{t}{x}$ by $t$ and $x$ by $\dfrac{1}{x}$, $t(t+x)\ldots(t+(\ell_w -1)x)\FT_{w}(x)$ becomes
    \[ t(t+1)\ldots(t+\ell_w-1)\FT_w^\rev(x). \]
    Thus, dividing both side of equation (\ref{eqn:better_identity}) by $x^n$ then replacing $\dfrac{t}{x}$ by $t$ and $x$ by $\dfrac{1}{x}$, we get the following equation
    \begin{equation}\label{eqn:reverse-identity}
        \sum_{w\in \SSS_n} t^{\ell_w}G^\rev_{w}(x) = \sum_{w\in\SSS_n}t(t+1)\ldots(t+\ell_w -1)\FT^\rev_{w}(x).
    \end{equation}
    
    \begin{example}\label{ex:rev-3}
        For $n = 3$, the LHS of (\ref{eqn:reverse-identity}) is
        \[ t^3 + 3t^2(2x + 1) + 2t(6x^2 + 6x + 1), \]
        and the RHS is
        \[ t(t+1)(t+2) + 3t(t+1)(2x) + 2t(6x^2+3x). \]
        One can check that they are equal.
    \end{example}

    A \textit{directed cycle} of size $k$, denoted $\DC_k$, is a graph on $k$ vertices $v_1,v_2,\ldots,v_k$ with directed edges $v_1 \rightarrow v_2 \rightarrow \ldots \rightarrow v_k \rightarrow v_1$. Note that the directed cycle of size $1$ is the graph $v_1 \rightarrow v_1$, and the directed cycle of size 2 is the graph $v_1\rightarrow v_2 \rightarrow v_1$. Recall that the number of tubings of $\DC_k$ with $i$ tubes is the coefficient of $x^i$ in $G^{\rev}_k(x)$.

    \begin{definition}\label{def:DC}
        We define $\DCl_{\ell,n}$ to be the set of all graph $G$ such that $G$ has $n$ vertices labelled $\{1,2,\ldots,n\}$ and is a disjoint union of $\ell$ directed cycles.
    \end{definition}
    
    We have the following interpretation for the coefficient of $t^\ell x^k$ in the LHS of (\ref{eqn:reverse-identity}).

    \begin{lemma}\label{lem:LHS-count}
        The coefficient of $t^\ell x^k$ in the LHS of (\ref{eqn:reverse-identity}) counts the number of tubings with $k$ tubes of graphs in $\DCl_{\ell,n}$.
    \end{lemma}

    \begin{proof}
        The term $t^\ell x^k$ only comes from permutations with $\ell$ cycles, i.e. it comes from the partial sum
        \[ \sum_{\substack{w\in\SSS_n \\ \ell_w = \ell}}t^\ell G^\rev_{w}(x). \] 

        For each permutation $w$ with cycle type $(\lambda_1,\ldots,\lambda_\ell)$, the cycles of $w$ determine the graph $\DCw_w$ in the canonical way: for each cycle $(w_{i_1}w_{i_2}\ldots w_{i_z})$, draw a directed cycle $w_{i_1} \rightarrow w_{i_2} \rightarrow \ldots \rightarrow w_{i_z} \rightarrow w_{i_1}$. Clearly, $\DCw_w$ is a disjoint union of $\ell$ directed cycles $\DC_{\lambda_1},\ldots,\DC_{\lambda_\ell}$ and has $n$ vertices labelled $\{1,2,\ldots,n\}$. Conversely, every graph in $\DCl_{\ell,n}$ is a graph $\DCw_w$ for some permutation $w\in\SSS_n$ with $\ell$ cycles by reversing the above construction. Thus,
        \[ \DCl_{\ell,n} = \bigcup_{\substack{w\in\SSS_n \\ \ell_w = \ell}}\DCw_w. \]

        Finally, in $G^\rev_w(x)$, $G^\rev_{\lambda_i}$ counts tubings in $\DC_{\lambda_i}$ by the number of tubes. Thus, the coefficients of $x^k$ in $G^\rev_w(x)$ counts tubings in $\DCw_w$ with $k$ tubes. Summing over all $w\in\SSS_n$ with $\ell_w = \ell$, we get the desired sum.
    \end{proof}

    \begin{example}\label{ex:DCEx}
        The coefficient of $tx$ in the LHS of Example \ref{ex:rev-3} is $12$. Figure \ref{fig:DCEx} shows the 12 tubings with one tube on graphs in $\DCl_{1,3}$. The graph in the two columns on the left corresponds to the permutation $(123)$, and the other graph corresponds to the permutation $(132)$.
        
        \begin{figure}[h!]
            \centering
            \includegraphics[scale = 0.5]{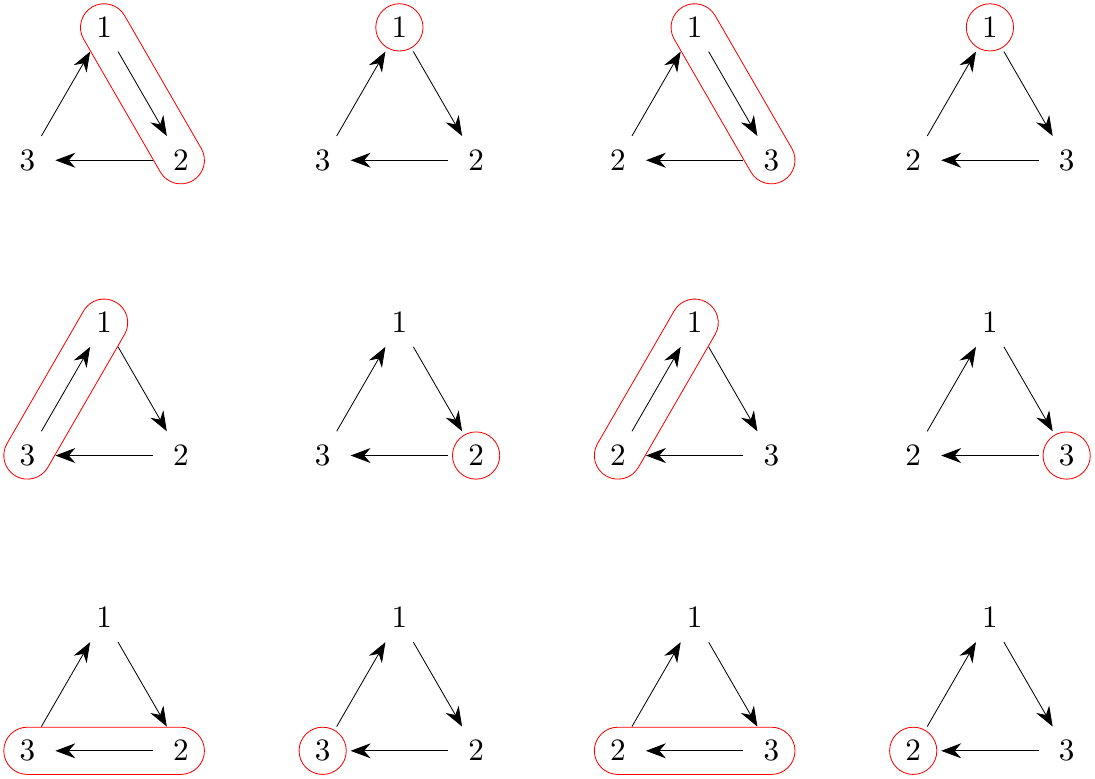}
            \caption{Tubings on graphs in $\DCl_{1,3}$}
            \label{fig:DCEx}
        \end{figure}
    \end{example}

    Now we move to the RHS of (\ref{eqn:reverse-identity}). A \textit{directed path} of size $k$, denoted $\DP_k$, is a graph on $k$ vertices $v_1,v_2,\ldots,v_k$ with directed edges $v_1 \rightarrow v_2 \rightarrow \ldots \rightarrow v_k$. Clearly, each directed path has a unique source and a unique sink. We say a tubing of $\DP_k$ is \textit{bottom-excluding} if is has the tube $\{v_2,\ldots,v_k\}$. Recall that the number of tubings of $\DP_{k-1}$ with $i$ tubes is the coefficient of $x^i$ in $F^\rev_{k-1}(x)$, so the number of bottom-excluding tubings of $\DP_k$ with $i$ tubes is the coefficient of $x^i$ in $xF^\rev_{k-1}(x)$. Similarly, for a graph $G$ that is a disjoint union of directed paths, we say a tubing of $G$ is bottom-excluding if the tubes in each directed path form a bottom-excluding tubing.

    \begin{definition}\label{def:DP}
        We define $\DPl_{\ell,n}$ to be the set of all graph $G$ such that $G$ has $n$ vertices labelled $\{1,2,\ldots,n\}$ and is a disjoint union of $\ell$ directed paths. Furthermore, for each permutation $w$, we define $\DPw_w$ to be the set of all graphs obtained from $\DCw_w$ by removing exactly one edge from each directed cycle.
    \end{definition}

    \begin{example}
        We have $\DCw_{(123)}$ is the graph with a directed cycle $1\rightarrow 2 \rightarrow 3 \rightarrow 1$. Then $\DPw_{(123)}$ is the set of three graphs: $1\rightarrow 2 \rightarrow 3$, $2 \rightarrow 3 \rightarrow 1$, and $3 \rightarrow 1\rightarrow 2$.
    \end{example}

    \begin{lemma}\label{lem:DP-count}
        Let $w$ be a permutation with cycle type $(\lambda_1,\lambda_2,\ldots,\lambda_\ell)$, then $\DPw_w$ has $\lambda_1\lambda_2\ldots\lambda_\ell$ graphs. Each graph in $\DPw_w$ is a disjoint union of $\DP_{\lambda_1}, \DP_{\lambda_2} \ldots, \DP_{\lambda_\ell}$. Furthermore,
        \[ \DPl_{\ell,n} = \bigcup_{\substack{w\in\SSS_n \\ \ell_w = \ell}} \DPw_w. \]
    \end{lemma}

    \begin{proof}
        The first statement follows from basic counting, and the second follows from the definition. For the last statement, clearly $\DPw_w \subset \DPl_{\ell,n}$ for all $w\in\SSS_n$ with $\ell$ cycles. On the other hand, for every graph $G\in\DPl_{\ell,n}$, adding a directed edge from the source to the sink of each directed path in $G$ gives $\DCw_w$ for some $w\in\SSS_n$ with $\ell$ cycles.
    \end{proof}
    
    We have the following interpretation for the coefficient of $t^\ell x^k$ in the LHS of (\ref{eqn:reverse-identity}).

    \begin{lemma}\label{lem:RHS-count}
        The coefficient of $t^\ell x^k$ in the RHS of (\ref{eqn:reverse-identity}) counts pairs of $(T, \sigma)$, where $T$ is a bottom-excluding tubing with $k$ tubes of some graph in $\DPl_{r,n}$ for some $r\geq\ell$, and $\sigma$ is a permutation in $\SSS_r$ with $\ell$ cycles.
    \end{lemma}

    \begin{proof}
        The term $t^\ell x^k$ in the RHS of (\ref{eqn:reverse-identity}) comes from permutations with $r\geq \ell$ cycles, i.e. it comes from the partial sum
        \[ \sum_{r\geq\ell}\sum_{\substack{w\in\SSS_n \\ \ell_w = r}}t(t+1)\ldots(t+r -1)\FT^\rev_{w}(x). \]
        
        For each permutation $w$ with cycle type $(\lambda_1,\ldots,\lambda_r)$, where $r\geq \ell$, recall that the coefficient of $t^{\ell}$ in $t(t+1)\ldots(t+r-1)$ counts the number of permutations in $\SSS_r$ with $\ell$ cycle. We claim that the coefficient of $x^k$ in $\FT^\rev_w(x)$ counts the number of bottom-excluding tubings with $k$ tubes of some graph in $\DPw_w$.
        
        Indeed, we rewrite $\FT^\rev_w(x)$ slightly as
        \[ \lambda_1\ldots\lambda_r x^r F^\rev_{\lambda_1 - 1}(x)\ldots F^\rev_{\lambda_r - 1}(x). \]
        Notice that $\lambda_1\ldots\lambda_r$ is the size of $\DPw_w$. For each graph $G\in\DPw_w$, $xF^\rev_{\lambda_i-1(x)}$ counts bottom-excluding tubings of $\DP_{\lambda_i}$ by the number of tubes. Hence, $x^r F^\rev_{\lambda_1 - 1}(x)\ldots F^\rev_{\lambda_r - 1}(x)$ counts bottom-excluding tubings of $G$ by the number of tubes. This means that for every $r\geq\ell$, the coefficient of $t^\ell x^k$ in
        \[ \sum_{\substack{w\in\SSS_n \\ \ell_w = r}}t(t+1)\ldots(t+\ell_w -1)\FT^\rev_{w}(x) \]
        counts pairs of $(T, \sigma)$, where $T$ is a bottom-excluding tubing with $k$ tubes of some graph in $\DPl_{r,n}$, and $\sigma$ is a permutation in $\SSS_r$ with $\ell$ cycles. Summing over all $r\geq\ell$, this completes the proof.
    \end{proof}

    \begin{example}\label{ex:DPEx}
        The coefficient of $tx$ in the RHS of Example \ref{ex:rev-3} is $12$. In Figure \ref{fig:DPEx}, the first and third columns show bottom-excluding tubings of graphs in $\DP_{1,3}$. They are paired with the only permutation in $\SSS_1$ with $1$ cycle, the identity permutation, giving $6$ pairs. The second and fourth columns show bottom-excluding tubings of graphs in $\DP_{2,3}$. They are paired with the only permutation in $\SSS_2$ with $1$ cycle, the $(12)$ permutation, giving another $6$ pairs.
        
        \begin{figure}[h!]
            \centering
            \includegraphics[scale = 0.46]{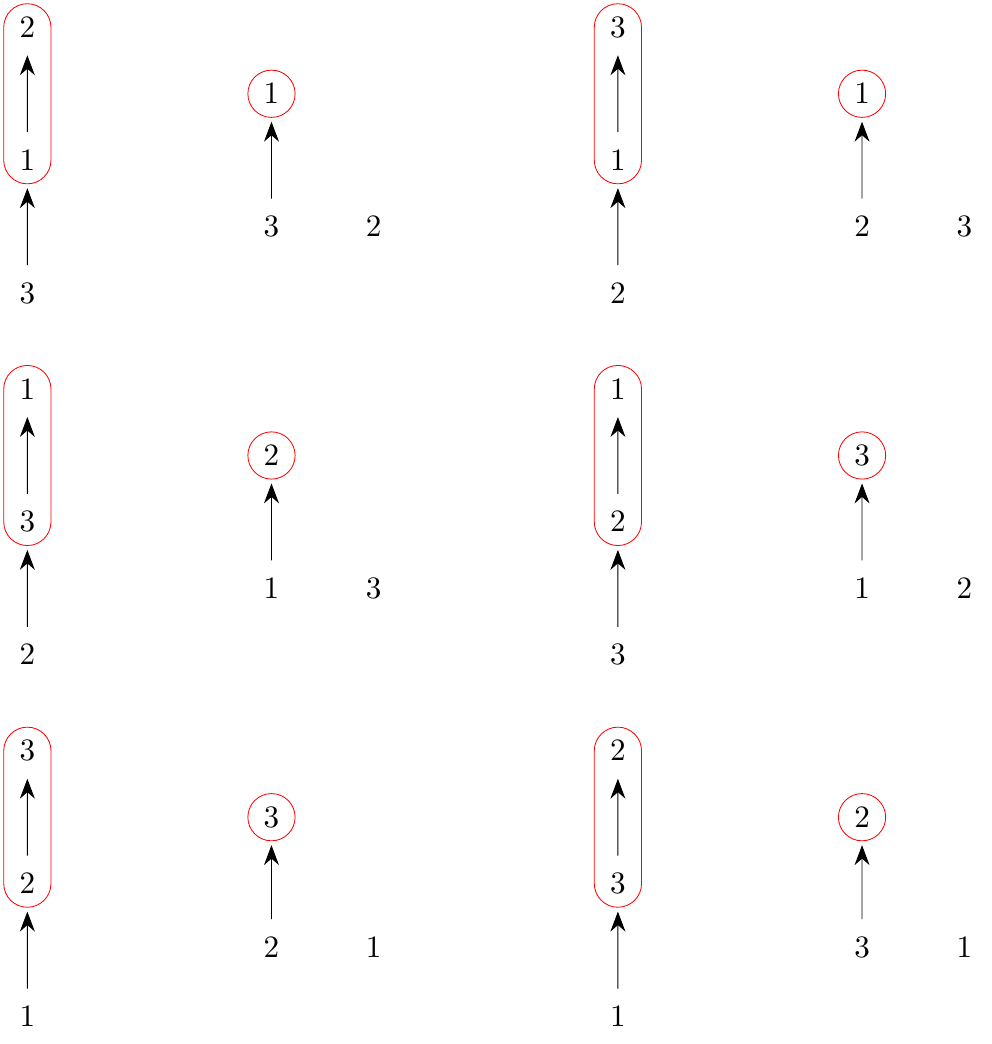}
            \caption{Bottom-excluding tubings on graphs in $\DP_{1,3}$ and $\DP_{2,3}$}
            \label{fig:DPEx}
        \end{figure}
    \end{example}

    Now we prove equation (\ref{eqn:reverse-identity}).

    \begin{prop}\label{prop:DC-DP-biject}
        The numbers of
        \begin{enumerate}
            \item tubings with $k$ tubes of graphs in $\DCl_{\ell,n}$; and
            \item pairs of $(T, \sigma)$, where $T$ is a bottom-excluding tubing with $k$ tubes of some graph in $\DPl_{r,n}$ for some $r\geq\ell$, and $\sigma$ is a permutation in $\SSS_r$ with $\ell$ cycles
        \end{enumerate}
        are the same.
    \end{prop}

    \begin{proof}
        We will construct a bijection between the two sets. Recall that we call a vertex \textit{lonely} if it is not in any tube, and a tube \textit{maximal} if it is not contained in any other tube. Given a tubing with $k$ tubes of a graph $G \in \DCl_{\ell, n}$, we construct a pair of $(T,\sigma)$ as follows.

        \begin{enumerate}
            \item  For each lonely vertex in $G$, remove the edge coming into it. This does not break connectivity of any tube, so we can keep the tubes the same. After this step, we have a tubing $T$ of some graph in $\DPl_{r,n}$ for some $r\geq\ell$.
            \item Order the directed paths in increasing order of their smallest vertices and construct $\sigma\in\SSS_r$ as follows: if in $G$ there is an arrow from the sink of the $i$th directed path to the source of the $j$th directed path, $\sigma_i = j$.
        \end{enumerate}

        First, we claim that $T$ is bottom-excluding. This is because by the definition of tubings, between every two consecutive maximal tubes in a (directed) cycle, there is at least one lonely vertex. Thus, there is a tube containing every vertex between two consecutive lonely vertices, unless they are next to each other. Hence, after removing the edges, in every directed path, there is a tube containing every vertex except the source. Thus, $T$ is bottom-excluding.

        Furthermore, there are exactly $\ell$ directed cycles in $G$, so the resulting permutation $\sigma$ has exactly $\ell$ cycles. Hence, the pair $(T,\sigma)$ satisfies the requirements.

        The inverse map is also straightforward. Given a pair $(T,\sigma)$, where $T$ is a bottom-excluding tubing with $k$ tubes of some graph $G'\in\DPl_{r,n}$ for some $r\geq\ell$, and $\sigma$ is a permutation in $\SSS_r$ with $\ell$ cycles, we first order the directed paths in $G'$ in increasing order of their smallest vertices. Then, we add an arrow from the sink of the $i$th directed path to the source of the $\sigma_i$th directed path and keep the tubes the same. Since $\sigma$ has $\ell$ cycles, the resulting graph is in $\DCl_{\ell,n}$. In addition, since $T$ is bottom-excluding, there cannot be adjacent maximal tubes in the resulting graph, so this is a valid tubing. This completes the proof.
    \end{proof}

    \begin{example}
        One can see examples of step 1 by checking tubings at the same position in Figure \ref{ex:DCEx} and \ref{ex:DPEx}.
    \end{example}

    \begin{example}
        Figure \ref{fig:tube-bijection} gives another example of the bijection.
        
        \begin{figure}[h!]
            \centering
            \includegraphics[scale = 0.35]{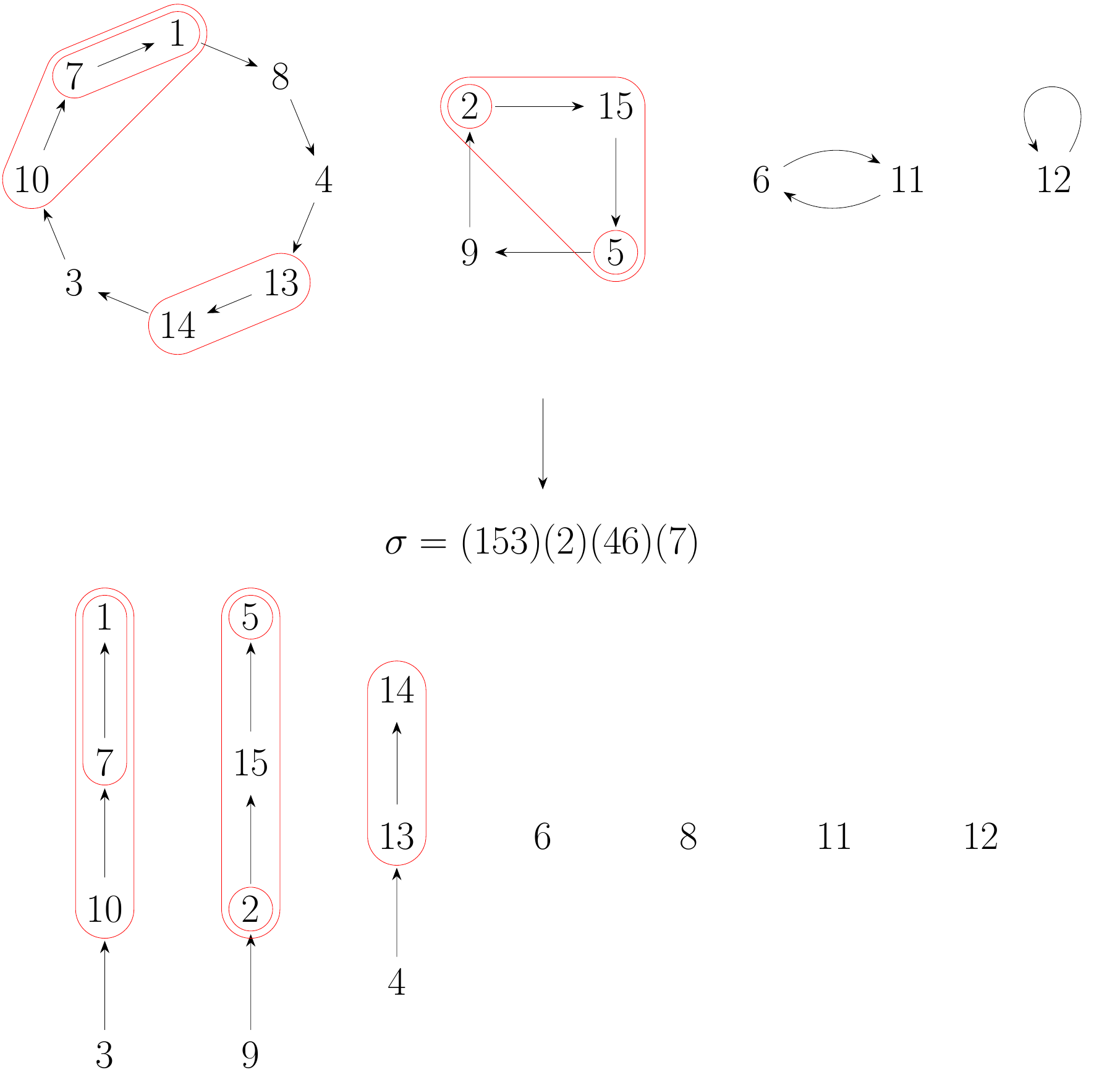}
            \caption{Example of the bijection}
            \label{fig:tube-bijection}
        \end{figure}
    \end{example}

\section{Narayana polynomials and Eulerian polynomials identities}\label{sec:corollary}

    When applying Theorem \ref{thm:recurrence} to \textit{broom posets}, one obtains several identities involving Eulerian polynomials and descent generating functions of stack-sorting preimages. Let us first recall relevant definitions.

    The Eulerian polynomial is defined to be
    \[ E_n(x) = \sum_{w\in\SSS_n} x^{\des(w)}. \]
    Let $\oplus$ denote the ordinal sum of posets. Then, $E_n(x)$ is the $h$-polynomial of $\AAA(A_1\oplus A_n)$ as well as $\AAA(A_i\oplus A_1\oplus A_j)$ where $i+j = n$. More generally, we define
    \[ E_{m,n}(x) = \sum_{\substack{w\in\SSS_n \\ w_1 \leq m,~ w_{m+n} \geq m+1}}x^{\des(w)}. \]
    It was found by Sack in his FPSAC 2023 Extended Abstract that $E_{m,n}(x)$ is the $h$-polynomial of $A_m\oplus A_n$.

    On the other hand, recall that the Narayana polynomial is defined to be
    \[ N_n(x) = \sum_{k=0}^{n-1}\dfrac{1}{n}\binom{n}{k}\binom{n}{k+1}x^k. \]
    Since $\AAA(C_{n+1})$ is the associahedra, $N_n(x)$ is the $h$-polynomial of $\AAA(C_{n+1})$. Thus, let $P = C_{n+1}$, and $P_i = A_{1}\oplus A_i$ for $1\leq i \leq n$, Theorem \ref{thm:recurrence} gives the following identity.

    \begin{cor}\label{cor:identity1}
        For all $n$,
        \[ N_n(x) = \dfrac{1}{n!}\sum_{w\in \SSS_n} B_w(x)E_{\ell_w}(x). \]
    \end{cor}

    A more general version of Corollary \ref{cor:identity1} involves \textit{stack-sorting}, an algorithm first introduced by Knuth in \cite{knuth1973art} that led to the study of pattern avoidance in permutations. The deterministic version, defined by West in \cite{west1990permutations}, is as follows. Given a permutation $w \in \SSS_n$, its \textit{stack-sorting image} $s(w)$ is obtained through the following procedure. Iterate through the entries of $w$. In each iteration,
    \begin{itemize}
        \item if the stack is empty or the next entry is smaller than the entry at the top of the stack, push the next entry to the top of the stack;
        \item else, pop the entry at the top of the stack to the end of the output permutation.
    \end{itemize}
    Permutations $w$ such that $s(w) = 12\ldots n$ are called \textit{stack-sortable permutations}, whose descent generating function is also the Narayana polynomial, i.e.
    \[ N_n(x) = \sum_{w\in s^{-1}(12\ldots n)}x^{\des(w)}. \]
    More generally, let $\SSS_{n,k} = \{w~|~w\in \SSS_{n+k}, w_i = i~\text{for all}~i>k\}$, then one can define
    \[ N_{n,k}(x) = \sum_{w\in\SSS_{n,k}} x^{\des(w)}. \]
    It was showed in \cite[Theorem 4.8]{nguyen2023stack} that $N_{n,k}(x)$ is the $h$-polynomial of the poset associahedra of the broom poset $A_{n,k} = C_{n+1}\oplus A_k$. Hence, let $P = C_{n+1+k}$ and $P_i = C_{n+1}\oplus A_i$ for $1\leq i \leq k$, we have the following identity.

    \begin{cor}\label{cor:identity2}
        For all $n,k$,
        \[ N_{n+k}(x) = \dfrac{1}{k!}\sum_{w\in \SSS_k} B_w(x)N_{n,\ell_w}(x). \]
    \end{cor}

    \begin{remark}
        In the case of $k = 2$ in Corollary \ref{cor:identity2}, we have
        \[ N_{n+2}(x) = \dfrac{1}{2}\left(N_{n,2}(x) + (1+x)N_{n+1}(x)\right). \]
        This is Proposition 5.3 in \cite{nguyen2023stack}.
    \end{remark}

    On the other hand, one can also let $P = C_{n+1}\oplus A_{k}$ and $P_i = A_i\oplus A_1 \oplus A_k$ for $1\leq i\leq n$, then one has the following identity.

    \begin{cor}\label{cor:identity3}
        For all $n,k$,
        \[ N_{n,k}(x) = \dfrac{1}{n!}\sum_{w\in \SSS_n} B_w(x)E_{\ell_w + k}(x). \]
    \end{cor}

    Combining Corollary \ref{cor:identity2} and \ref{cor:identity3}, we have the following.

    \begin{cor}\label{cor:identity4}
        For all $n,k$,
        \[ N_{n+k}(x) = \dfrac{1}{n!k!}\sum_{\pi\in \SSS_n, \sigma\in\SSS_k} B_\pi(x)B_\sigma(x)E_{\ell_\pi+\ell_\sigma}(x). \]
    \end{cor}

    Combining Corollary \ref{cor:identity1} and \ref{cor:identity4} gives the following.

    \begin{cor}\label{cor:identity5}
        For all $n,k$,
        \[ \dfrac{1}{(n+k)!}\sum_{w\in\SSS_{n+k}}B_w(x)E_{\ell_w}(x) = \dfrac{1}{n!k!}\sum_{\pi\in \SSS_n, \sigma\in\SSS_k} B_\pi(x)B_\sigma(x)E_{\ell_\pi+\ell_\sigma}(x). \]
        In other words,
        \[ \sum_{w\in\SSS_{n+k}}B_w(x)E_{\ell_w}(x) = \binom{n+k}{n}\sum_{\pi\in \SSS_n, \sigma\in\SSS_k} B_\pi(x)B_\sigma(x)E_{\ell_\pi+\ell_\sigma}(x). \]
    \end{cor}

    On the other hand, one can modify the context of Corollary \ref{cor:identity3} slightly: let $P = C_n\oplus A_k$, and $P_i = A_i\oplus A_k$ for $1\leq i \leq n$. This gives the following.

    \begin{cor}\label{cor:identity6}
        For all $n,k$,
        \[ N_{n-1,k}(x) = \dfrac{1}{n!}\sum_{w\in \SSS_n} B_w(x)E_{\ell_w, k}(x). \]
    \end{cor}

    Similar to Corollary \ref{cor:identity4}, we also have the following.

    \begin{cor}\label{cor:identity7}
        For all $n,k$,
        \[ N_{n+k-1}(x) = \dfrac{1}{n!k!}\sum_{\pi\in \SSS_n, \sigma\in\SSS_k} B_\pi(x)B_\sigma(x)E_{\ell_\pi,\ell_\sigma}(x). \]
    \end{cor}

    Now we focus on a special case. Let $P = C_{n+1}\oplus A_k$, $P_1 = C_n\oplus A_k$ and $P_2$ be the \textit{two-leg broom poset} $A_2\oplus C_{n-1}\oplus A_k$. Then, we have
    \[ 2h_P(x) - (1+x)h_{P_1}(x) = h_{P_2}(x), \]
    or
    \[ 2N_{n,k}(x) - (1+x)N_{n-1,k}(x) = h_{P_2}(x). \]
    Recall that $h_P(x)$ and $h_{P_1}(x)$ count descents in $s^{-1}(\SSS_{n,k})$ and $s^{-1}(\SSS_{n-1,k})$, respectively. Thus, we have the following proposition.

    \begin{prop}
        The $h$-polynomial of $\AAA(A_2\oplus C_{n-1}\oplus A_k)$ counts descents in
        \[ \{w\in s^{-1}(\SSS_{n+1,k})~|~w_1\leq n+k-1, w_{n+k+1}\geq n+k\}. \]
    \end{prop}

    \begin{proof}
        Let $H_1(x)$ be the descent generating function of $\{w\in s^{-1}(\SSS_{n+1,k})~|~w_1\leq n+k-1, w_{n+k+1} = n+k+1\}$ and $H_2(x)$ be that of $\{w\in s^{-1}(\SSS_{n+1,k})~|~w_1\leq n+k-1, w_{n+k+1}= n+k\}$.

        Observe that counting descents in $\{w\in s^{-1}(\SSS_{n+1,k})~|~w_1\leq n+k-1, w_{n+k+1} = n+k+1\}$ is the same as in $\{w\in s^{-1}(\SSS_{n,k})~|~w_1< n+k\}$. Counting descents in $\{w\in s^{-1}(\SSS_{n,k})~|~w_1= n+k\}$ is the same as in $s^{-1}(\SSS_{n-1,k})$ with an extra descent at the beginning. Hence,
        \[ H_1(x) = N_{n,k}(x) - xN_{n-1,k}(x). \]
        The same argument applies for $\{w\in s^{-1}(\SSS_{n+1,k})~|~w_1\leq n+k-1, w_{n+k+1}= n+k\}$ with the caveat that when $w_{n+k} = n+k+1$, there is an extra descent at the end. Thus,
        \[ H_2(x) = N_{n,k}(x) - (x+1-x)N_{n-1,k}(x). \]
        Then,
        \[ H_1(x) + H_2(x) = 2N_{n,k}(x) - (1+x)N_{n-1,k}(x) = h_{\AAA(A_2\oplus C_{n-1}\oplus A_k)}(x). \]
    \end{proof}

    \begin{remark}
        The set $\{w\in s^{-1}(\SSS_{n+1,k})~|~w_1\leq n+k-1, w_{n+k+1}\geq n+k\}$ bears resemblance to the definition of $E_{n+k-1,2}(x)$. Thus, one may hope for a combinatorial formula for the $h$-polynomial of $\AAA(A_j\oplus C_n\oplus A_k)$ that interpolates between $N_{n,k}(x)$ and $E_{j,k}(x)$.
    \end{remark}

\bibliography{bibliography}
\bibliographystyle{alpha}

\end{document}